\documentclass{amsart}
\usepackage{amsmath}
\usepackage{amscd}
\usepackage{verbatim}
\usepackage{amsfonts}
\usepackage{amsthm}
\usepackage{amssymb}
\usepackage{graphicx}
\usepackage[cmtip,all]{xy}
\usepackage{url}

\newcommand{\lt}{\triangleleft}

\newcommand{\Z}{\mathbb{Z}}
\newcommand{\R}{\mathbb{R}}

\newtheorem{theorem}{Theorem}[section]

\newtheorem{corollary}[theorem]{Corollary}

\theoremstyle{definition}
\newtheorem{definition}[theorem]{Definition}
\newtheorem{example}[theorem]{Example}

\theoremstyle{remark}

\numberwithin{equation}{section}

\begin{document}

\title{Twist Spinning Knotted Trivalent Graphs}
\author{J. Scott Carter, Seung Yeop Yang}

\date{}

\begin{abstract}



In 1965, E. C. Zeeman proved that the $(\pm1)$-twist spin of any knotted sphere in $(n-1)$-space is unknotted in the $n$-sphere. In 1991, Y. Marumoto and Y. Nakanishi gave an alternate proof of Zeeman's theorem by using the moving picture method. In this paper, we define a knotted $2$-dimensional foam which is a generalization of a knotted sphere and prove that a $(\pm1)$-twist spin of a knotted trivalent graph may be knotted. We then construct some families of knotted graphs for which the  $(\pm1)$-twist spins are always unknotted.

\end{abstract}

\maketitle

\section{Introduction}

The paper will begin with a short history of higher dimensional knot theory. We present the basic definitions and facts  necessary for this study in Section $1$. We quickly review twist spinning as a construction of knotted spheres in Section $2$.  These ideas are generalized to construct the twist spin of a knotted trivalent graph in Section $3$. There we study $(\pm 1)$-twist spins. 
We work in the smooth (or piecewise linear) category throughout this paper.

\subsection{Historical Background} The birth of higher dimensional knot theory occurred in 
 1925,  when Artin~\cite{Art25} introduced the notion of spinning of a classical knot which is a method of making a knotted surface from a given classical knot. In the late 1950s and early 1960s, a number of authors, including Haefliger~\cite{Hae62}, Kinoshita~\cite{Kin60}, Levine~\cite{Lev65}, and Gluck~\cite{Glu62}, looked at the study of higher dimensional knots from the point of view of algebraic topology. In 1961, Fox~\cite{Fox61} presented a talk at Georgia Topology Conference that evolved into the paper ``A quick trip through knot theory." He gave interesting examples of knotted spheres and the fundamental groups of their complements.  
 
 The fundamental group is an algebraic quantity associated to the complement of the knot or knotted sphere. The most basic invariant of a knot or knotted sphere is the complement itself. In the classical case of knotted circles in $3$-dimensional space, the complement is a complete invariant. However, this result, due to Gordon and Luecke~\cite{Gor89}, is rather recent. It has been known that links are not determined by their complement. In the case of knotted spheres in $4$-space, Gluck~\cite{Glu62} proved that there are distinct knotted spheres that have homeomorphic complements.

In 1965, Zeeman~\cite{Zee65} introduced the notion of twist spun knots and showed that $(\pm1)$-twist spun knots are equivalent to trivial $2$-knots. More generally twist spun knots are fibered with the fiber being the branched covers of the punctured $3$-sphere with branch set the given knot. In 1979, Litherland~\cite{Lit79} generalized Zeeman's construction by defining  \emph{deform spinning}. In a letter to Cameron Gordon, Litherland proved that Fox's example $12$  is the $2$-twist spun trefoil, and subsequently Kanenobu~\cite{Kan83} showed that all of Fox's $2$-spheres are twist spun knots.

\begin{center}
\rule{2in}{0.01in}
\end{center}

A \emph{knotted(spatial) trivalent graph} or simply \emph{KTG} is an embedding into $3$-space of a graph, for which every vertex has degree $3$. In this paper we will study the notion of $n$-twist spin of a knotted trivalent graph as a generalization of Artin and Zeeman's studies \cite{Art25, Zee65}. The twist spins of a KTG are examples of knotted foams (this concept will be defined in a moment). This paper is a study of a limited class of examples of knotted foams; even so the results are surprising enough to indicate that the study of knotted foams is potentially interesting.

\section{Knotted $2$-Dimensional Foams ($2$-Foams)}

\subsection{Preliminaries}
A \emph{knotted sphere} or simply \emph{$2$-knot} is the image of a smooth (or PL locally flat) embedding $f:S^{2} \to \mathbb{R}^{4}$. A \emph{knotted $2$-dimensional foam} is analogous to a  knotted sphere in the same way that a knotted trivalent graph is analogous to a classical knot. A formal definition follows.

\begin{definition}
A \emph{$2$-dimensional foam} or simply \emph{$2$-foam} is a compact topological space such that a neighborhood of any point is homeomorphic to a neighborhood of a point in $Y^2$ that is depicted in  Fig.~\ref{1}.
\end{definition}

\begin{figure}[h]
\includegraphics[width=6cm]{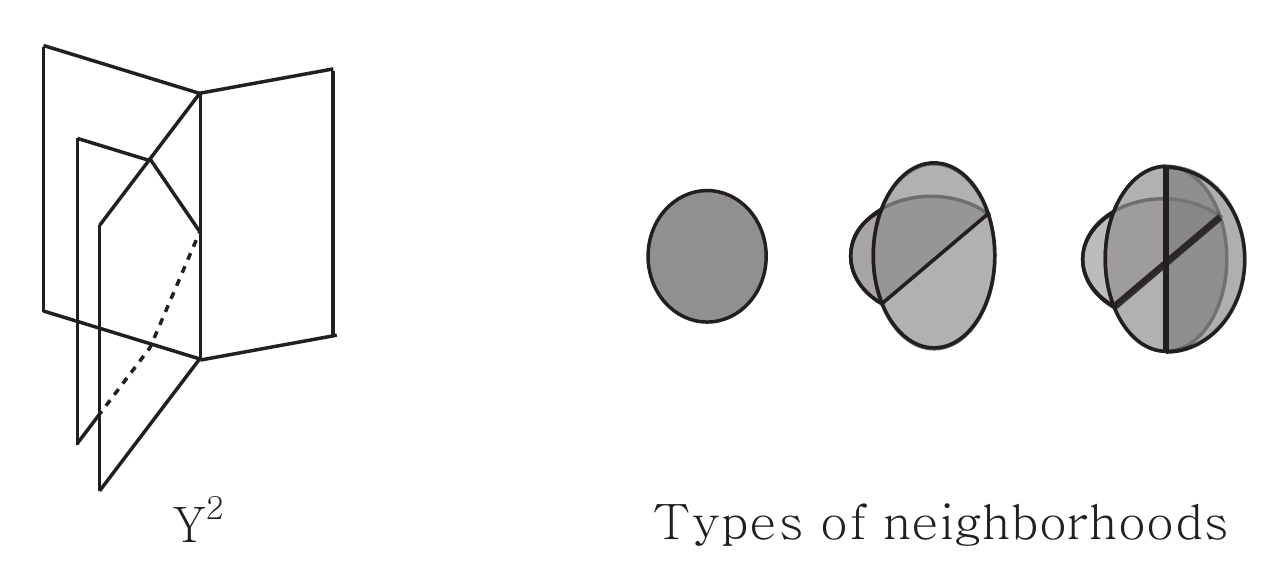}
\caption{Neighborhoods of a point in $Y^{2}$}
\label{1}
\end{figure}

For more specificity, we define $Y_1^0 =(0,1/2,1/2)$, $Y^0_2=(1/2,0,1/2)$, and $Y^0_3=(1/2,1/2,0)$. The space $Y^1$ is the cone on $Y_1^0 \cup Y^0_2 \cup Y^0_3$ where the cone point is the barycenter $(1/3,1/3,1/3)$ of the standard equilateral triangle
$\Delta^2= \{ (x,y,z) \in \R^3: 0 \le x,y,z \ \&  \ x+y+z=1 \}$. Let 
$$\Delta^3= \{ (x_1,x_2,x_3,x_4) \in \R^4 : \sum_{i=1}^4 x_i=1  \ \& \ 0\le x_i \ {\mbox{\rm for}} \ i=1,2,3,4 \}.$$ Let $\Delta^2_i$ denote the $i\/$th face in which $x_i=0$. Consider a copy $Y^1_i \subset \Delta^2_i$ for each $i=1,2,3,4.$ The space $Y^2$ is the cone on $\cup_{i=1}^4 Y^1_i$ where the cone point is the baycenter $b=(1/4,1/4,1/4,1/4)$. Clearly, the space $Y^2$ generalizes to all dimensions. In such generalizations, one can construct notions of $n$-dimensional foams. 

According to the definition of a foam, there are three types of neighborhoods of a point in a $2$-foam (Fig.~\ref{1}), and the boundary of a $2$-foam is a trivalent graph.  A $2$-foam with no boundary is said to be \emph{closed}. 

A $2$-foam is an analogue of a trivalent graph. So 
aspects of embeddings of trivalent graphs will be reviewed. An embedded knotted graph $\Gamma$ in $\mathbb{R}^{3}$ (or $S^{3}$) is \emph{unknotted} if it is ambiently isotopic to an embedding in $\mathbb{R}^{2}$. A subgraph of $\Gamma$ which is a knot in $\mathbb{R}^{3}$ is called a \emph{constituent knot} of $\Gamma$, and a mutually disjoint union of constituent knots of $\Gamma$ is said to be a \emph{constituent link} of $\Gamma$.

Any $\theta$-curve embedded in $\mathbb{R}^{3}$ has three constituent knots. Although all three constituent knots may be unknotted on a $\theta$-curve, the $\theta$-curve may be knotted. For example, the Kinoshita's $\theta$-curve (in Fig.~\ref{8}) is a knotted $\theta$-curve for which all constituent knots are unknotted. A graph $\Gamma$ in $\mathbb{R}^{3}$ is said to be \emph{almost trivial} if $\Gamma$ is a non-trivial planar graph and every proper subgraph of $\Gamma$ is unknotted \cite{Kaw96}.

\begin{figure}[h]
\includegraphics[width=2.5cm]{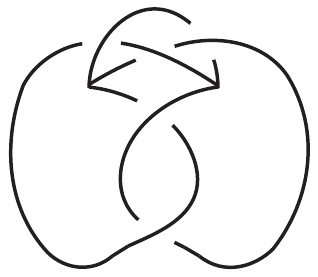}
\caption{Kinoshita's $\theta$-curve}
\label{8}
\end{figure}

\begin{definition} The {\it $k$-twist spin} of  a classical knot $K$ is obtained as follows.
Let $\R^3_+= \{(x,y,z): 0\le z\}$, and  for $\theta \in [0, 2\pi), $ let $\R^3_+(\theta)= \{(x,y,z \cos{\theta}, z \sin{\theta}): 0\le z \}$. Consider once and for all a standard identification between $\R^3_+$ and $\R^3_+(\theta).$ Embed a $(1,1)$ tangle whose closure is $K$ into the ball $B=\{(x,y,z): x^2+y^2+(z-3)^2 \le 1 \}$ such that the end points of the tangle are at $(0,\pm 1, 3)$. Now join these by {\sf L}-shaped arcs to the points $(0,\pm 3, 0)$ in the plane $z=0$; see Fig.~\ref{Tw} for details. As $\theta$ varies, rotate the space in the ball $B$ a total of $k$ full twists. Since $\R^4$ can be thought of the union of $\R^3_+(\theta)$ as $\theta$ varies and since the tangle twists within the ball $B$, the result is a sphere embedded in $\R^4$. The half spaces $\R^3_+(\theta)$ are called the {\it pages} of {\it an open book decomposition} of $\R^4$.

Note that an $n$-twist spin of a KTG is defined similarly, but the graph is broken along one edge and embedded in the ball $B$ in a way similar to the tangle representation of $K$. \end{definition}

\begin{figure}[h]
\includegraphics[width=8cm]{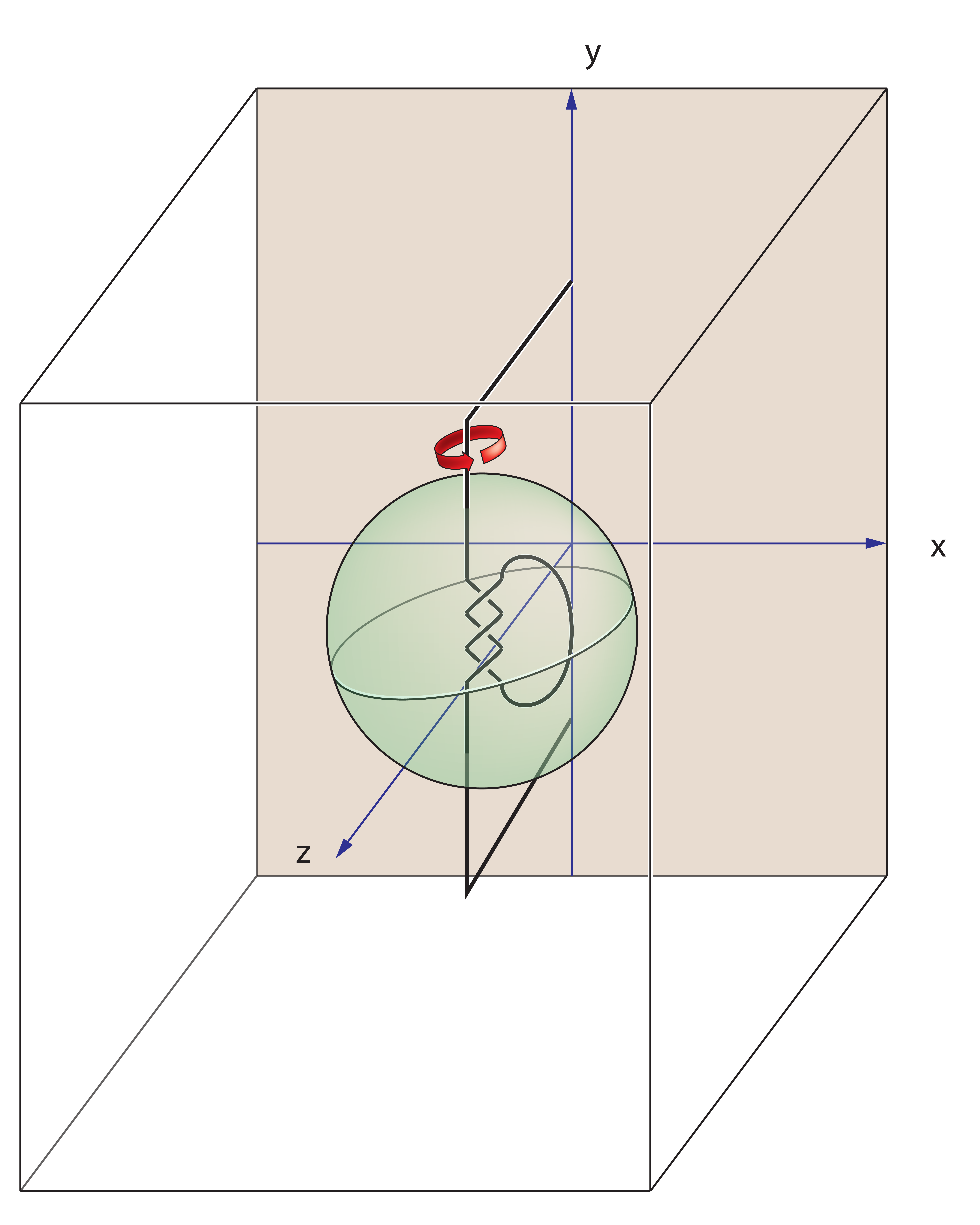}
\caption{Twist spinning a classical knot}
\label{Tw}
\end{figure}

We now discuss the analogues of triviality for $2$-foams. We say that an embedded $2$-foam in $4$-space is \emph{unknotted} if it is ambiently isotopic to an embedding in $3$-space.

\subsection{Is a $(\pm1)$-twist spin of a knotted trivalent graph unknotted?}
The following theorem due to Zeeman is remarkable, and the motivation for studying twist spins of KTGs. First we need to develop some notation. 

A knot is called \emph{fibered} if its complement is a fiber bundle over $S^{1}$. If a $2$-knot is fibered, its fiber is a $3$-manifold whose boundary is homeomorphic to $S^{2}$. The closed $3$-manifold obtained from such a fiber by attaching a $3$-ball along its boundary is called the \emph{closed fiber}.

\begin{theorem}\cite{Zee65} \label{Zman}
The $(\pm k)$-twist spin of any knot $K \subset S^{n-1}$ is fibered   where the fiber is the $k$-fold branched covering of $S^{n-1}$ in which the branch set is the knot $K$.
\end{theorem}

In the case of the $(\pm 1)$-twist spin of a knot, the $1$-fold branched cover of $S^{3}$ branched along the knot is tautologically $S^{3}$.  Thus the closed fiber is a $3$-sphere, So in this case the resulting knotted surface is trivial. 

When we spin a knotted trivalent graph it is surprising that the result can be non-trivial. The following theorem is one of our main results.

\begin{theorem}\label{main}
There is a nontrivial foam obtained by a $(\pm1)$-twist spinning of a knotted $\theta$-curve. That is, a $(\pm1)$-twist spun KTG is not always unknotted.
\end{theorem}

We remind the reader (Chapter 9 \cite{Kam02}) that a knotted surface can berepresented in a normal form as a classical knot diagram with rectangular bands that are labeled $u$ or $d$ (for up or down). The bands represent saddle points of a critical level embedding and are resolved as indicated in Fig.~\ref{bandaid}. After all up or down resolutions, the resulting knot diagrams are unlinks. 

\begin{figure}[h]
\includegraphics[width=3cm]{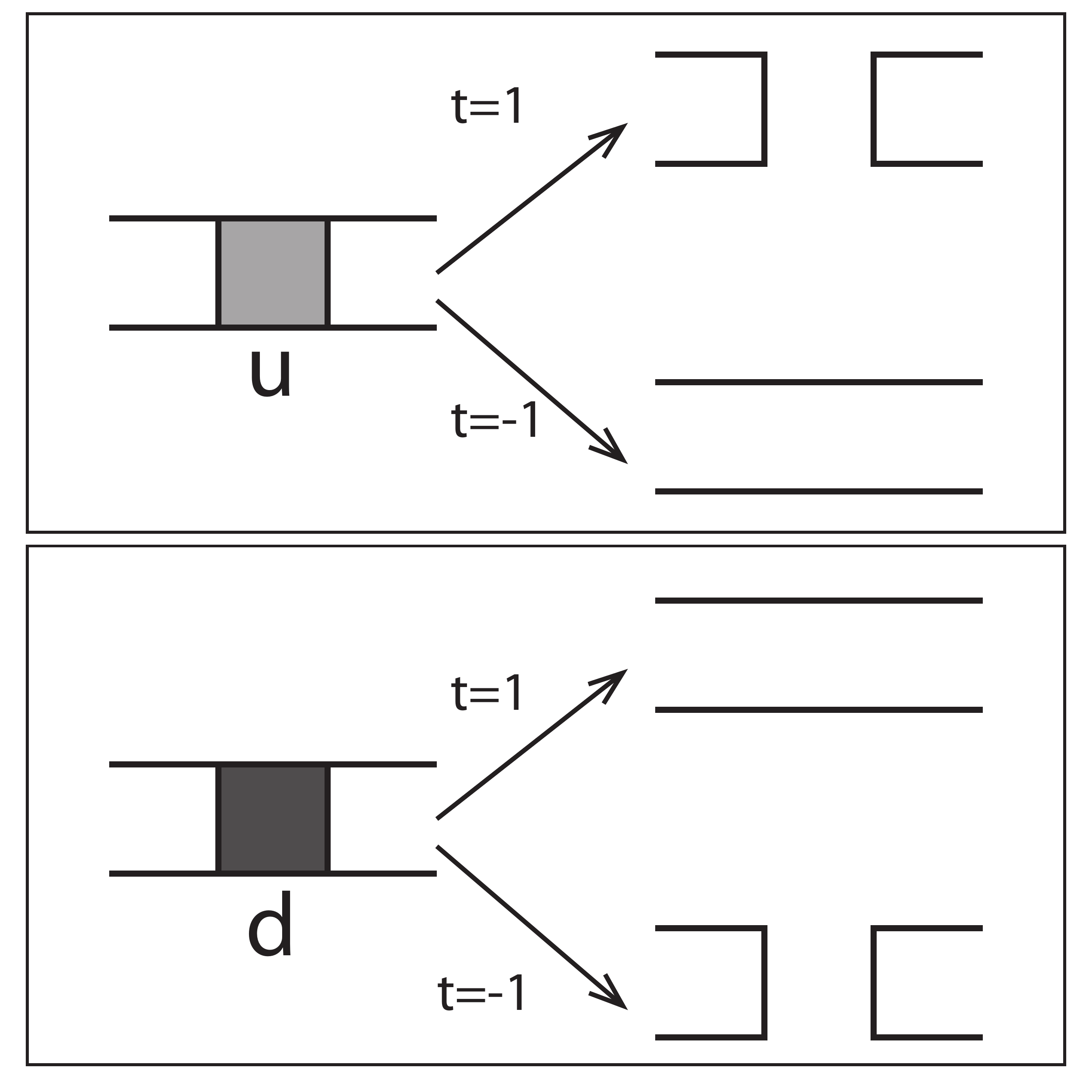}
\caption{Resolving up and down bands}
\label{bandaid}
\end{figure}

In 1991, Y. Marumoto and Y. Nakanishi \cite{MarNak91} gave an alternate proof of the triviality of $(\pm 1)$-twist spins by using the motion picture method. Since we will use their idea to prove our main theorem, we will look at their proof in some detail.


By Zeeman's twist spinning method, we can construct the $(\pm1)$-twist spin of $K$, denoted by $t_{\pm1}(K)$, from the given classical knot $K$. Since all of knots can be deformed into a trivial knot by performing a finite number of crossing changes, we only need to show that $t_{\pm1}(K)$ and $t_{\pm1}(K^{'})$ are ambiently isotopic when $K^{'}$ is the result of changing a crossing of $K$.

Let $*$ be the crossing at which the crossing change will occur. Pull the crossing close to the plane $\mathbb{R}^{2}$ that is fixed during spinning as indicated in case 1 of Fig.~\ref{2}. The crossing then becomes a minimal point. If the crossing is opposite, then we can deform this crossing to the first case crossing by using the second Reidemeister move as in case 2 of Fig.~\ref{2}. Then the first case's and second case's diagrams are locally the same nearby the fixed plane.

\begin{figure}[h]
\includegraphics[width=8cm]{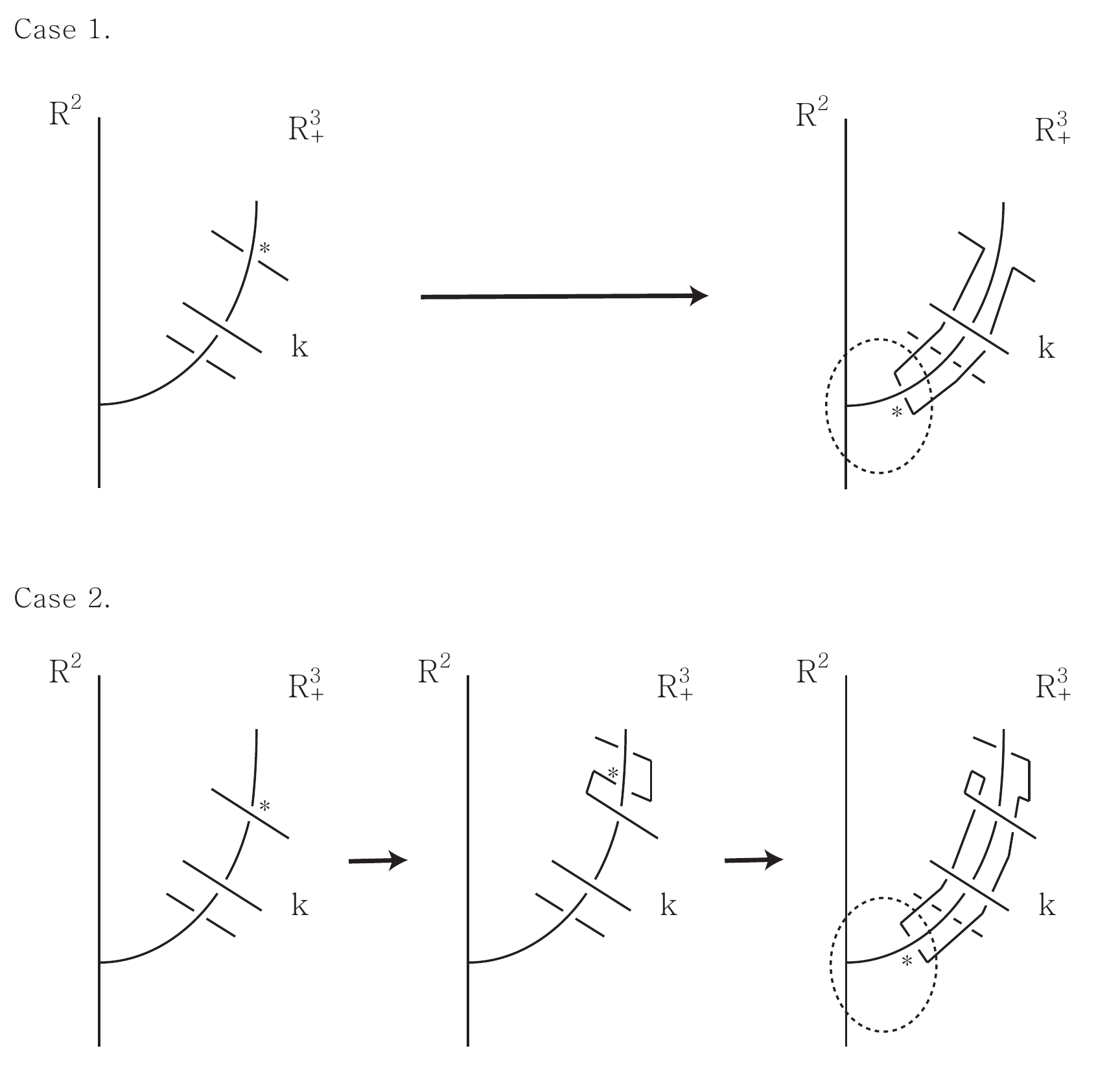}
\caption{Marumoto and Nakanishi's idea I}
\label{2}
\end{figure}

\begin{figure}
\includegraphics[width=12cm]{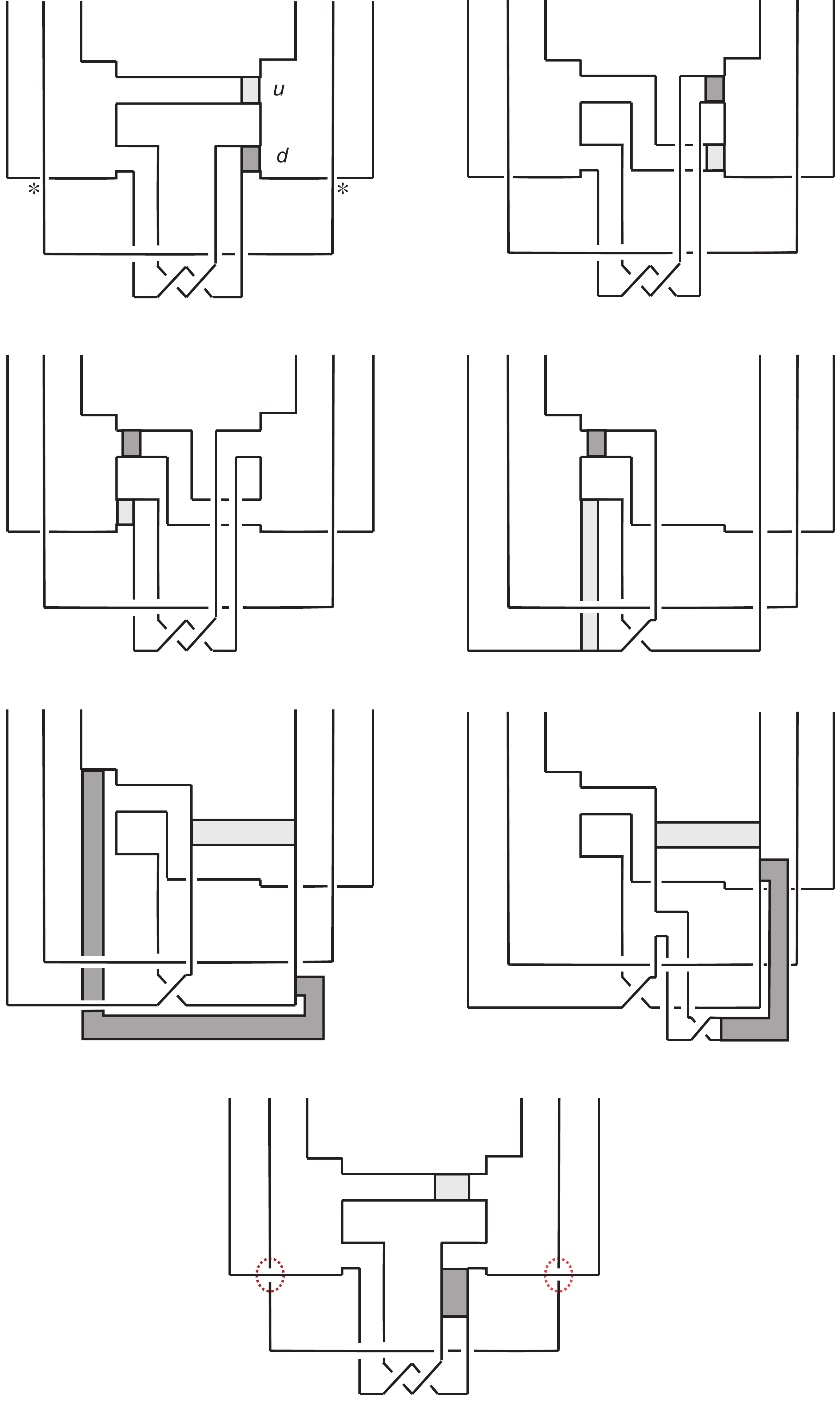}
\caption{Marumoto and Nakanishi's idea II}
\label{3}
\end{figure}

Then the normal form of the $(\pm1)$-twist spin of the diagram nearby the fixed plane is the first picture in Fig.~\ref{3}. The first picture and the second picture are ambiently isotopic. In the second picture, if we push the upper saddle band to the level of lower saddle band and the lower saddle band to the level of upper saddle band, then the saddle bands may switch their positions, and the saddle bands now can stretch up and down. We can deform the third picture to the last one by an ambient isotopy of $\mathbb{R}^{4}$. In the last normal form in Fig.~\ref{3}, we see that the crossings $*$ are changed. Therefore the crossing changes can be realized by an ambient isotopy of $\mathbb{R}^{4}$.

We next consider an important difference between spun knots and spun KTGs. Consider a knot $K$ in $S^{3}$, and take a small $3$-dimensional ball $B^{3}$ on the knot $K$. Then we can consider the closure of $S^{3} \setminus B^{3}$ as one page in an open book decomposition. Since a knot $K$ is an $1$-dimensional manifold, the $3$-ball on the knot $K$ can travel anywhere on the knot $K$. So spinning the classical knot does not depend on which arc that is split; 
see Fig.~\ref{4}. But since a KTG has vertices, the $3$-ball may not move anywhere on the KTG. \emph{Therefore spinning the KTG  depends upon the splitting arc.}

\begin{figure}[h]
\includegraphics[width=6.0cm]{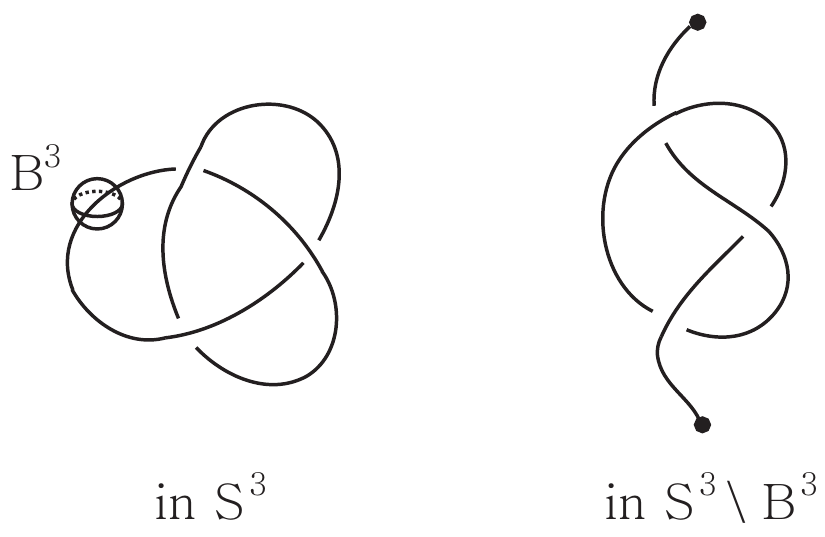}
\caption{Spinning the knot does not depend on the splitting arc.}
\label{4}
\end{figure}

When a KTG is spun or twist spun, it has constituents that are analogous to the constituents of a KTG. A constituent is either a knotted sphere or a knotted torus.

\begin{theorem}
Every foam obtained by a $(\pm1)$-twist spin of a knotted $\theta$-curve has exactly two trivial $S^{2}$-constituent components.
\end{theorem}

\begin{proof}
Let $\Gamma$ be a knotted $\theta$-curve. Label its edges $e_{1}$, $e_{2}$, and $e_{3}$. There are three constituent knots of $\Gamma$; these are $e_{1} \cup e_{2}$, $e_{2} \cup e_{3}$, and $e_{3} \cup e_{1}$. Choose an edge $e_{i}$ of the knotted $\theta$-curve $\Gamma$, cut at a point on $e_{i}$, and attach the end points to the plane $\mathbb{R}^{2}$. Then by using Marumoto and Nakanishi's technique, we can deform this broken KTG to satisfy the following conditions: (See Fig.~\ref{7})

\begin{enumerate}
\item There are no crossings on the broken edges $e_{i}$.
\item For each $j \neq i$, $e_{j}$ is an unknotted edge.
\end{enumerate}

Denote this deformed KTG by $\widetilde{G}$ and the $(\pm1)$-twist of $\widetilde{G}$ with respect to $e_{i}$ by $t_{e_{i}}(\widetilde{G})$. Then the $(\pm1)$-twist spins of $e_{1} \cup e_{2}$, $e_{2} \cup e_{3}$, and $e_{3} \cup e_{1}$ are constituent components of $t_{e_{i}}(\widetilde{G})$. Since $e_{i} \cup e_{j}$ is a unknotted arc for each $j\neq i$, $(\pm1)$-twist spinning $e_{i} \cup e_{j}$ makes a trivial $S^{2}$-constituent component of $\Gamma$.
\end{proof}

\begin{figure}[h]
\includegraphics[width=4cm]{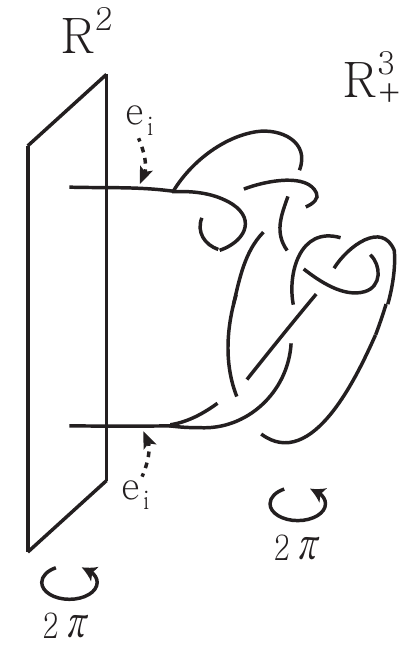}
\caption{A deformed broken KTG}
\label{7}
\end{figure}

\begin{example}
Let $\Gamma$ be the knotted $\theta$-curve depicted in  Fig.~\ref{5}. We take the edge $e_{1}$ of $\Gamma$ and construct a deformed KTG with respect to the edge $e_{1}$ as in the following steps. We first pull the edge $e_{3}$ close to the plane and change the crossing by using Marumoto and Nakanishi's technique. This process gives the second diagram. The second, third, and fourth pictures are ambiently isotopic. In the forth picture, we see that $(\pm1)$-twist spinning $e_{1} \cup e_{2}$ and $e_{1} \cup e_{3}$ construct two trivial $S^{2}$-constituent components.
\end{example}

\begin{figure}[h]
\includegraphics[width=10cm]{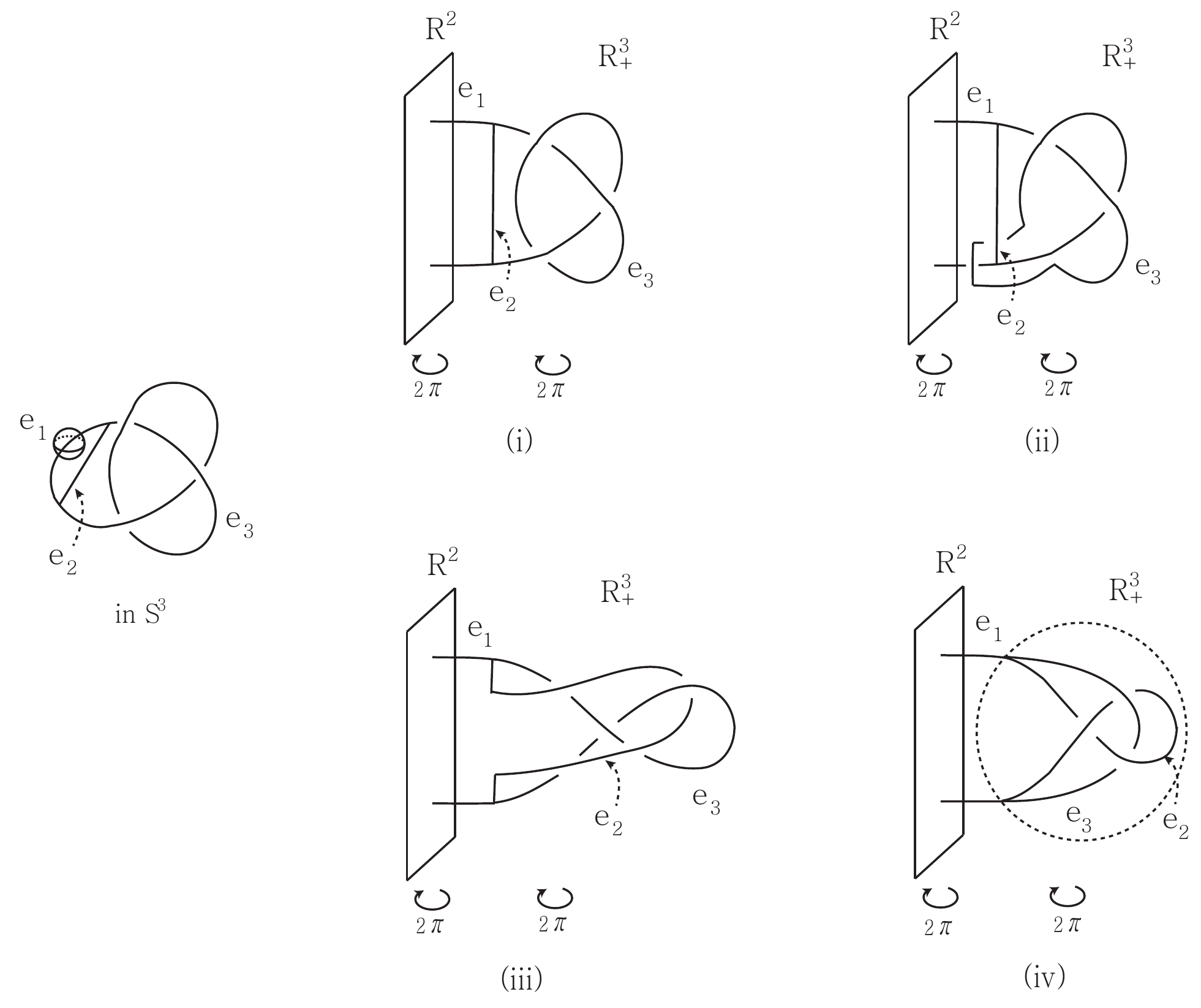}
\caption{A $(\pm1)$-twist spin of a knotted $\theta$-curve has exactly two trivial $S^{2}$-constituent components.}
\label{5}
\end{figure}

We are going to prove that there is a KTG such that a $(\pm1)$-twist spin of the KTG is a knotted foam in $\mathbb{R}^{4}$. First we need to recall from \cite{IIJO} the definition of a  $G$-family of quandles.

\begin{definition} Let $G$ denote a group and $X$ a set. Suppose that for each $g\in G$, there is a binary operation $\lt_g:X \times X \rightarrow X$ that satisfies the properties:
\begin{itemize}
\item $x \lt_g x =x$ for all $x\in X$ and $g\in G$,
\item $(x \lt_g y) \lt_h y= x \lt_{gh} y$ for all $x,y \in X$ and $g,h\in G$,
\item $x\lt_1 y = x$ for all $x,y \in X$ where $1\in G$ denotes the identity, and
\item $(x \lt_ g y) \lt_h z = (x \lt_h z) \lt_{{h^{-1} g h}} (y \lt_h z)$ for all $x,y,z \in X$ and $g,h\in G$.
\end{itemize} Then $(X, \{ \lt_g: g\in G \})$ is called a {\it $G$-family of quandles}. In this case there is a quandle structure on $X\times G$ that is defined by $(x,g)\lt(y,h)= (x \lt_h y, h^{-1}g h)$. This structure is called {\it the associated quandle of a $G$-family of quandles.} \end{definition}

The concept of $G$-family of quandles was invented to color handlebody knots and KTGs, but the notion of coloring extends easily to colorings of knotted foams.
When a foam is given, for example, as a movie of a KTG, then the coloring condition is that each still in the movie is colored as a KTG, and that the colorings obey consistency through Reidemeister moves. 

Here we are only interested in the colorings by the quandle $Q=\Z/3 \times \Z_{2}$, with quandle operation 
$$(x, \epsilon) \lt (y,1)=  (2y-x, \epsilon) \quad (x,\epsilon) \lt (y, 0)= (x, \epsilon).$$ 
We think of the set $\{0,1,2\} = \Z/3$ as the set of colors $\{R,G,B\}$ which are stand-ins for {\it Red, Green}, and {\it Blue.} The elements of $Q$ are of the form $(C, 0)$ or $(C,1)$ where $C\in \{R,G,B\}$. For each color $C$, abuse notation and  let $C:=(C,1)$ and $c:=(C,0)$. In our diagrams, an arc that is colored with $(C,0)=c$ is indicated as a dotted arc, and those colored by $(C,1)=C$ are indicated by solid arcs. 

The coloring conditions for trivalent vertices and crossings are as follows:
\begin{itemize}
\item the first component of the colors $(C,\epsilon)$  at a vertex are all the same;
\item either all three arcs are dotted, or exactly one is dotted;
\item If a solid arc with color $(C,1)$ crosses over another arc of a different color (whether or not the under arc be solid or dotted), the colors change on the under-arc;
\item the colors on an under-arc remain unchanged if crossed over by a dotted arc or if crossed over by a solid arc of the same color. \end{itemize}
Figure~\ref{coloringcond} indicates these conventions.  

\begin{figure}[h]
\includegraphics[width=6cm]{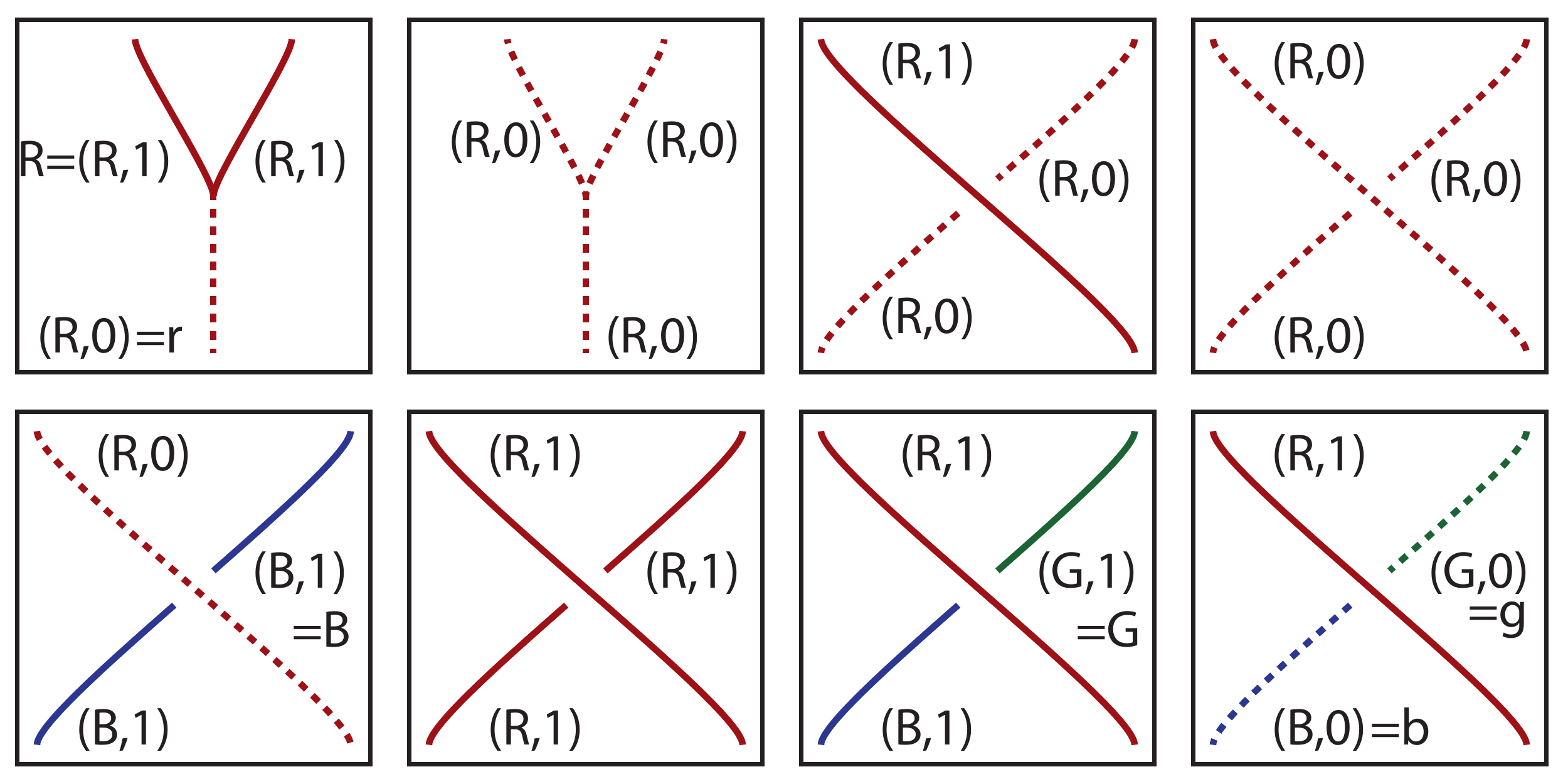}
\caption{Coloring conditions at crossings}
\label{coloringcond}
\end{figure}

\begin{theorem}
If there exists a non-trivial coloring of the given foam by the associated quandle of a $G$-family of quandles, then the foam is knotted.
\end{theorem}

\begin{proof} Such a coloring gives rise to a non-trivial representation of the fundamental group of the complement of the foam into the associated group of the quandle.  
\end{proof}

\begin{theorem}
There is a nontrivial foam obtained by a $(\pm1)$-twist spinning of a knotted $\theta$-curve. That is, a $(\pm1)$-twist spun KTG is not always unknotted.
\end{theorem}

\begin{proof}
We consider the knotted $\theta$-curve $\Gamma$ in Fig.~\ref{5}. We choose the edge $e_{1}$ of $\Gamma$ to construct a $1$-twist spun $\Gamma$ with respect to the edge $e_{1}$. The $1$-twist spinning process can be illustrated schematically as Fig.~\ref{6}, and we can find a non-trivial coloring of $1$-twist spun $\Gamma$ by the associated quandle $Q=\Z/3\times \mathbb{Z}_{2}$. Similarly, a $(-1)$-twist spun $\Gamma$ with respect to the edge $e_{1}$ also has a non-trivial coloring. Therefore $(\pm1)$-twist spinning of $\Gamma$ with respect to the edge $e_{1}$ is a knotted foam.
\end{proof}

\begin{figure}[h]
\includegraphics[width=8cm]{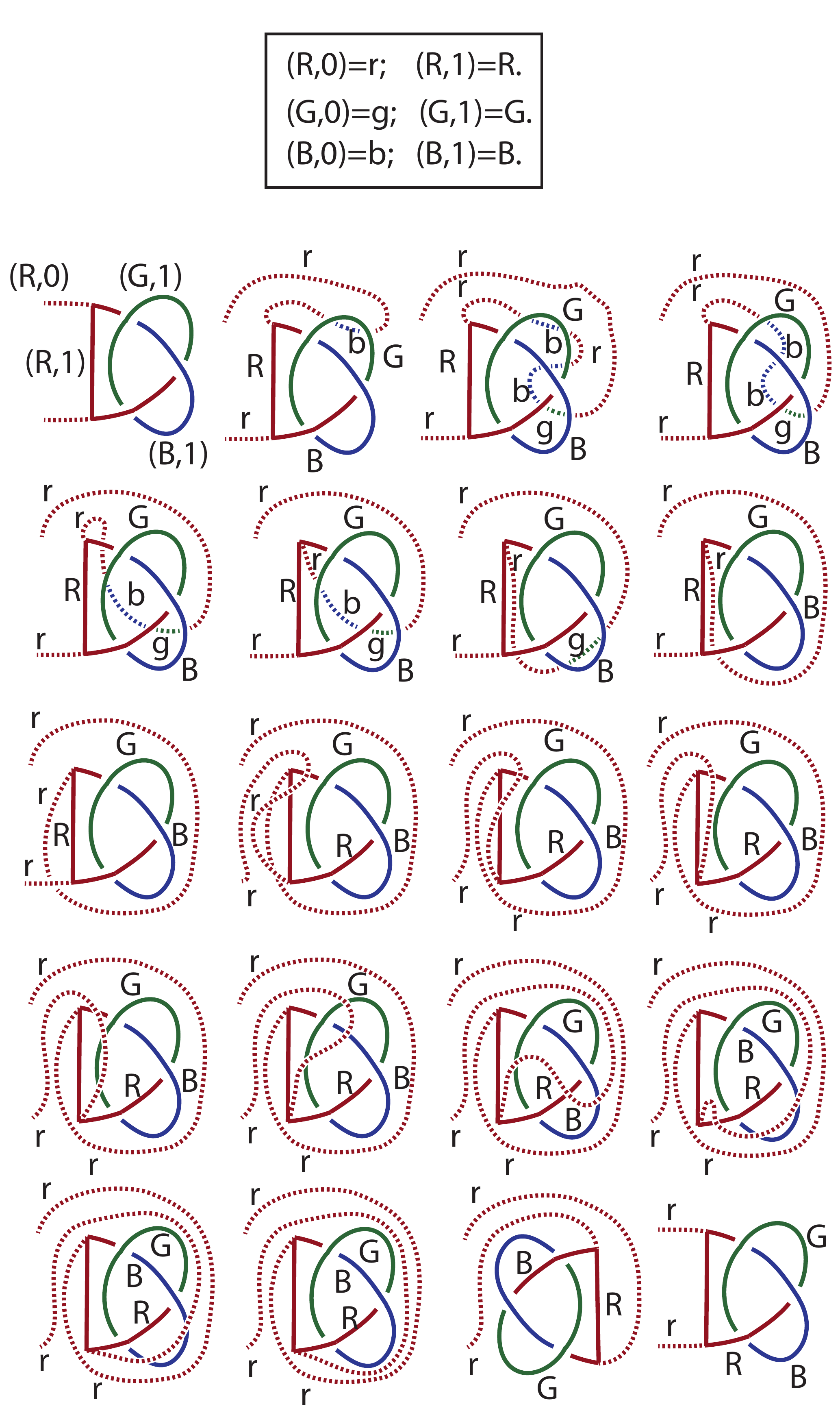}
\caption{A $(\pm1)$-twist spun KTG is not always unknotted.}
\label{6}
\end{figure}

\begin{theorem}
\label{AT}
Let $\Gamma$ be a spatial graph and $e$ an edge of $\Gamma$. Assume that $\Gamma \setminus \{e\}$ is unknotted. Then the $(\pm1)$-twist spin of $\Gamma$ with respect to the edge $e$ is unknotted.
\end{theorem}

\begin{proof}
To construct $(\pm1)$-twist spun $\Gamma$ with respect to the edge $e$, we cut the edge $e$, and pull the endpoints close to the plane. When pulling the endpoints we can make the endpoints cross over the edges transversely but not cross over the vertices since $\Gamma$ is finite. Now, choose one of the endpoints, trace on the edge until we come to a crossing. If we trace on the under-strand at this crossing, then we change the crossing. Otherwise we do not. Keep doing this process until we arrive at a vertex. Next, we apply the same strategy to the another endpoint. Note that these crossing changes do not deform the $(\pm1)$-twist spun $\Gamma$ by Y. Marumoto and Y. Nakanishi's technique. Denote the deformed KTG after these crossing changes by $\widetilde{\Gamma}$. Since only the crossings with respect to the edge $e$ are changed in this process, $\widetilde{\Gamma} \setminus \{e\}$ is still unknotted. So, we can deform $\widetilde{\Gamma}$ to be unknotted by using Reidemeister moves and Y. Marumoto and Y. Nakanishi's technique if we need. The $(\pm1)$-twist spin of $\widetilde{\Gamma}$ with respect to $e$, therefore, is unknotted, so the $(\pm1)$-twist spin of $\Gamma$ with respect to $e$ is also unknotted. The proof is complete.
\end{proof}

\begin{corollary}
Any $(\pm1)$-twist spin of an unknotted or almost trivial graph is unknotted.
\end{corollary}

\begin{proof}
If a graph is unknotted, then clearly $(\pm1)$-twist spin of the given graph is unknotted. Let $\Gamma$ be an almost trivial graph. Since $\Gamma$ is almost trivial, the subgraph $\Gamma \setminus \{e\}$ of $\Gamma$ is unknotted for each edge $e$. By Theorem~\ref{AT}, the $(\pm1)$-twist spin of $\Gamma$ with respect to the edge $e$ is unknotted for each edge $e$. Hence any $(\pm1)$-twist spin of $\Gamma$ is unknotted.
\end{proof}

By above corollary, we can easily check the following example.

\begin{example}
Any $(\pm1)$-twist spin of Kinoshita's $\theta$-curve is unknotted.
\end{example}

We do not know if the converse: 
 ``If any $(\pm1)$-twist spin of a $\theta$-curve is unknotted, then the $\theta$-curve is either unknotted or almost trivial," of above corollary holds. 
 
\begin{theorem}
If any $(\pm1)$-twist spin of a $\theta$-curve $\Gamma$ is unknotted and $\Gamma \setminus \{e\}$ has a nontrivial Fox coloring for each $e$, then the $\theta$-curve is either unknotted or almost trivial.
\end{theorem}

\begin{proof}
Suppose that a $\theta$-curve $\Gamma$ is knotted and is not almost trivial. Then there exists an edge $e$ such that $\Gamma \setminus \{e\}$ is knotted. Since $\Gamma \setminus \{e\}$ is $n$-colorable for some positive integer $n$, the $(\pm1)$-twist spin of $\Gamma$ with respect to $e$ is colored by the associated quandle $\Z/n\times \mathbb{Z}_{2}$ in a manner similar to Fig.~\ref{6}. This is a contradiction.
\end{proof}

We can also construct a knotted $\theta$-curve such that every $(\pm1)$-twist spin of the knotted $\theta$-curve is knotted. Since every constituent knot of the knotted $\theta$-curve in Fig.~\ref{9} is the connected sum of a positive trefoil knot and itself (a granny knot) which is $3$-colorable, so the $(\pm1)$-twist spin is knotted.

\begin{figure}[h]
\includegraphics[width=3.5cm]{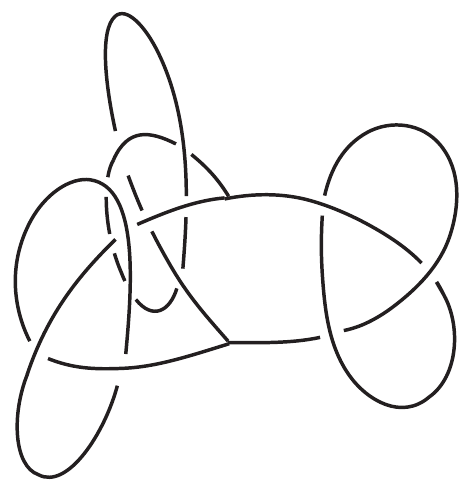}
\caption{A knotted $\theta$-curve for which every $(\pm1)$-twist spin is knotted.}
\label{9}
\end{figure}

\section*{Acknowledgements} The authors would like to thank Seiichi Kamada for some helpful advice. Yongju Bae introduced the authors of this paper. The paper grew out of the second author's master's thesis at the University of South Alabama.

\bibliographystyle{amsplain}

\begin{flushleft}
J. Scott Carter \\  
University of South Alabama \\ 
Mobile, AL 36688 \\
E-mail address: {\tt carter@southalabama.edu} 
\end{flushleft}

\begin{flushleft}
Seung Yeop Yang \\ 
The George Washington University \\ 
Washington, DC 20052 \\
E-mail address: {\tt  syyang@gwu.edu}
\end{flushleft}

\end{document}